\documentclass{amsart}
\usepackage[latin1]{inputenc}
\usepackage[T1]{fontenc}
\usepackage{amsmath}
\usepackage{amsthm}

\usepackage{bbold}
\usepackage{color}
\usepackage{url}
\usepackage{float}
\usepackage{graphicx}

\usepackage{algorithmic}

\floatstyle{ruled}
\newfloat{algorithm}{thp}{loa}
\floatname{algorithm}{Algorithm}

\newtheorem{theorem}{Theorem}[section]
\newtheorem{example}[theorem]{Example}
\newtheorem{proposition}[theorem]{Proposition}

\renewcommand\k{\mathbb{k}}

\DeclareMathOperator\lm{lm}
\DeclareMathOperator\lc{lc}
\DeclareMathOperator\lt{lt}

\title{A parallel Buchberger algorithm for multigraded ideals}
\author{Emil Sköldberg \and Mikael Vejdemo-Johansson \and Jason Dusek}

\begin{document}

\begin{abstract}
  We demonstrate a method to parallelize the computation of a Gröbner
  basis for a homogenous ideal in a multigraded polynomial ring. Our
  method uses anti-chains in the lattice $\mathbb N^k$ to separate
  mutually independent S-polynomials for reduction.
\end{abstract}

\maketitle

\section{Introduction}
\label{sec:introduction}
In this paper we present a way of parallelizing the Buchberger
algorithm for computing Gr\"obner bases in the special case of
multihomogeneous ideals in the polynomial algebra over a field.
We describe our algorithm as well as our implementation of
it. We also present experimental results on the efficiency of our
algorithm, using the ideal of commuting matrices as illustration.

\subsection{Motivation}
\label{subsec:motivation}
Most algorithms in commutative algebra and algebraic geometry at some
stage involve computing a Gr\"obner basis for an ideal or module.
This ubiquity together with the exponential complexity of the
Buchberger algorithm for computing Gr\"obner bases of homogeneous
ideals explains the large interest in improvements of the basic
algorithm.

\subsection{Prior Work}
\label{subsec:prior_work}
Several approaches have been tried in the literature. Some authors,
such as Chakrabarti--Yelick~\cite{155350} and Vidal~\cite{680307} have
constructed general algorithms for distributed memory and shared
memory machines respectively. Reasonable speedups were achieved on
small numbers of processors. Another approach has been using
factorization of polynomials; all generated S-polynomials are
factorized on a master node, and the reductions of its factors are
carried out on the slave nodes. Work by Siegl~\cite{siegl},
Bradford~\cite{96965}, Gr\"abe and Lassner~\cite{graebe_lassner}.
In a paper by Leykin~\cite{10.1109/ICPPW.2004.1328011} a coarse
grained parallelism was studied that was implemented both in the
commutative and non-commutative case.

Good surveys of the various approaches can be found in papers by
Mityunin and Pankratiev~\cite{MR2464549} and Amrhein,
B{\"u}ndgen and K{\"u}chlin~\cite{MR1624631}. Mityunin and Pankratiev
also give a theoretical analysis of and improvements to algorithms
known at that time.

Finally, an approach by Reeves~\cite{MR1635246} parallelizes on the
compiler level for modular coefficient fields.


\subsection{Our approach}
\label{subsection:our_approach}
Our approach restricts the class of Gröbner bases treated to
homogenous multigraded Gröbner bases. While certainly not general
enough to handle all interesting cases, the multigraded case covers
several interesting examples. For these ideals we describe a coarsely
grained parallelization of the Buchberger algorithm with promising
results.

\begin{example}\label{commutator}
  Set $R=k[x_1,\dots,x_{n^2}, y_1,\dots,y_{n^2}]$ where $k$ is a
  field. Let $X$ and $Y$ be
  square $n\times n$-matrices with entries the variables
  $x_1,\dots,x_{n^2}$ and $y_1,\dots,y_{n^2}$ respectively. Then the
  entries of the matrix
  \[
  I_n = XY - YX
  \]
  form $n^2$ polynomials generating the ideal $I_n$ $R$.

  The computation of a Gröbner basis for $I_1$ and $I_2$ is trivial
  and may be carried out on a blackboard. A Gröbner basis for $I_3$ is
  a matter of a few minutes on most modern computer systems, and
  already the computation of a Gr\"obner basis for $I_4$ is
  expensive using the standard reverse lexicographic
  term order in $R$; the Macaulay2 system~\cite{M2} several hours are
  needed to obtain a Gr\"obner basis with 563 elements. However, using
  clever product orders, Hreinsd\'ottir has been able to find bases with
  293 and 51 elements~\cite{MR1324496,MR2251816}. As far as we are
  aware of, a Gr\"obner basis for $I_5$ is not known.

  By assigning multidegrees $(1,0)$ to all the variables
  $x_1,\dots,x_{n^2}$ and $(0,1)$ to all the variables
  $y_1,\dots,y_{n^2}$, the ideal $I_n$ becomes
  multigraded over $\mathbb N\times\mathbb N$, and thus approachable
  with our methods.
\end{example}

\begin{example}\label{computationaltopology}
  While this paper presents only the approach to multigraded ideals in
  a polynomial ring, an extension to free multigraded modules over
  multigraded rings is easily envisioned, and will be dealt with in
  later work.

  Gröbner bases for such free modules would be instrumental in
  computing invariants from applied algebraic topology such as the
  \emph{rank invariant} as well as more involved normal forms for
  higher dimensional persistence modules.\cite{multiDpersistence}
\end{example}

\section{Partially ordered monoids}
\label{sec:part-order-mono}

We shall recall some definitions and basic facts about partially
ordered sets that will be of fundamental use in the remainder of this
paper.

A partially ordered set is a set equipped with a binary, reflective,
symmetric and transitive \emph{order} operation $\leq$. Two objects
$a,b$ such that either $a\leq b$ or $b\leq a$ are called
\emph{comparable}, and two objects for which neither $a\leq b$ nor
$b\leq a$ are called \emph{incomparable}. A subset $A$ of a partially
ordered set in which all objects are mutually incomparable is called
an \emph{antichain}. An element $p$ is \emph{minimal} if there are no
distinct $q$ with $q\leq p$, \emph{maximal} if there are no distinct
$q$ with $p\leq q$, \emph{smallest} if all other $q$ fulfill $p\leq q$
and \emph{largest} if all other $q$ fulfill $q\leq p$.

There is a partially ordered set structure on $\mathbb N^d$ in which
$(a_1,\dots,a_d) \leq (b_1,\dots,b_d)$ iff $a_i\leq b_i$ for all
$i$. This structure is compatible with the monoid structure on
$\mathbb N^d$ in the sense that if $\bar a \leq \bar b$, then $\bar
c*\bar a\leq\bar c*\bar b$ for $\bar a, \bar b, \bar c\in\mathbb
N^d$. If a monoid has a partial order compatible with the
multiplication in this manner, we call it a \emph{partially ordered
  monoid}.

In a partially ordered set $P$, we say that a subset $Q$ is an
\emph{ideal} if it is downward closed, or in other words if for any
$p\in P, q\in Q$ such that $p\leq q$, then $p\in Q$. It is called a
\emph{filter} if it is upward closed, or if for $p\in P, q\in Q$ such
that $q\leq p$, then $p\in Q$.

An element $p$ is maximal in an ideal if any element $q$ such that
$p\leq q$ is not a member of the ideal. Minimal elements of filters
are defined equivalently. An ideal (filter) is \emph{generated} by its
maximal (minimal) elements in the sense that the membership condition
of the ideal (filter) is equivalent to being larger than (smaller
than) at least one generator. Generators of an ideal or filter form an
antichain. Indeed, if these were not an antichain, two of them would
be comparable, and then one of these two would not be
maximal (minimal). An ideal (filter) is \emph{finitely generated} if it has
finitely many generators, and it is \emph{principal} if it has exactly
one generator.

There is a partially ordered monoid structure on $\mathbb N^d$, given
by $(p_1,\dots,p_d)\leq (q_1,\dots,q_d)$ if $p_i\leq q_i$ for all $i$,
and by $(p_1,\dots,p_d)*(q_1,\dots,q_d) =
(p_1+q_1,\dots,p_d+q_d)$. This structure will be the main partially
ordered monoid in us in this paper.

\section{Multigraded rings and the grading lattice}
\label{sec:mult-rings-grad}

A polynomial ring $R=\k[x_1,\dots,x_r]$ over a field $\k$ is said to
be \emph{multigraded} over $P$ if each variable $x_j$ carries a
\emph{degree} $|x_j|\in P$ for some partially ordered monoid $P$.  We
expect of the partial order on $P$ that if $p, q\in P$ then $p\leq
p*q$ and $q\leq p*q$. The degree extends from variables to entire
monomials by requiring $|mn| = |m|*|n|$ for monomials $m, n$; and from
thence a multigrading of the entire ring $R$ follows by decomposing
$R=\bigoplus_{p\in P} R_p$ where $R_p$ is the set of all
\emph{homogenous} polynomials in $R$ of degree $p$, i.e.\@ polynomials
with all monomials of degree $p$. A homogenous polynomial of degree
$(n_1,\dots,n_d)$ is said to be of \emph{total degree}
$n_1+\dots+n_d$. We note that for the $\mathbb N^d$-grading on $R$,
the only monomial with degree $(0,0,\dots,0)$ is $1$, and thus the
smallest degree is assigned both to the identity of the grading monoid
and to the identity of the ring.

We write $\lm p$, $\lt p$, $\lc p$ for the leading monomial, leading
term and leading coefficient of $p$.

\begin{proposition}\label{prop:gradingdivides}
  Suppose $p$ and $q$ are homogenous. If $|p|\not\leq|q|$ then $\lm p$
  does not divide $\lm q$.
\end{proposition}
\begin{proof}
  If $\lm p | \lm q$ then $\lm q = c\lm p$, and thus $|\lm q| =
  |c|*|\lm p|$, and thus, since $|1|\leq|c|$, by our requirement for a partially ordered
  monoid, $\lm p\leq\lm q$.
\end{proof}

\begin{proposition}\label{prop:reductiondepends}
  Reduction in the Buchberger algorithm of a given multidegree for a
  homogenous generating set depends
  only on its principal ideal in the partial order of degrees.
\end{proposition}
\begin{proof}
  We recall that the reduction of a polynomial $p$ with respect to
  polynomials $q_1,\dots,q_k$ is given by computing
  \[
  p' = p-\frac{\gcd(\lm p, \lm q_j)}{\lm p}q_j
  \]
  for a polynomial $q_j$ such that $\lm q_j|\lm p$. We note that by
  Proposition \ref{prop:gradingdivides}, this implies $|\lm
  q_j|\leq|\lm p|$ and thus $|q_j|\leq|p|$.
\end{proof}

We note that Proposition \ref{prop:reductiondepends} implies that if
two S-polynomials are incomparable to each other, then their
reductions against a common generating set are completely independent
of each other. Furthermore, since $|p'| = |p|$, in the notation of the
proof of Proposition \ref{prop:reductiondepends}, a reduction of an
incomparable S-polynomial can never have an effect on the future
reductions of any given S-polynomial.

Hence, once S-polynomials have been generated, their actual reductions
may be computed independently across antichains in the partial order of
multidegrees, and each S-polynomial only has to be reduced against the
part of the Gröbner basis that resides below it in degree.

\section{Algorithms}
\label{sec:algorithms}

The arguments from Section \ref{sec:mult-rings-grad} lead us to an
approach to parallelization in which we partition the S-polynomials
generated by their degrees, pick out a minimal antichain, and
generate one computational task for each degree in the antichain.

One good source for minimal antichains, that is guaranteed to produce
an antichain, though most often will produce more tasks than are
actually populated by S-polynomials is to consider the minimal total
degree for an unreduced S-polynomial, and produce as tasks the
antichain of degrees with the same total degree.

Another, very slightly more computationally intense method is to take
all minimal occupied degrees. These, too, form an antichain by
minimality, and are guaranteed to only yield as many tasks as have
content.

Either of these suggestions leads to a master-slave distributed
algorithm as described in pseudocode in Algorithms \ref{alg:master} and
\ref{alg:slave}.

The resulting master node algorithm can be seen in Algorithm
\ref{alg:master}, and the simpler slave node algorithm in Algorithm \ref{alg:slave}.
\begin{algorithm}
\begin{algorithmic}
\LOOP
    \IF {have waiting degrees and waiting slaves}
        \STATE \texttt{nextdeg} $\leftarrow$ \texttt{pop(waiting degrees)}
        \STATE \texttt{nextslave} $\leftarrow$ \texttt{pop(waiting slaves)}
        \STATE send \texttt{nextdeg} to \texttt{nextslave}
    \ELSIF {all slaves are working}
        \STATE wait for message from slave \texttt{newslave}
        \STATE \texttt{push(newslave, waiting slaves)}
    \ELSIF {no waiting degrees and some working slaves}
        \STATE wait for message from slave \texttt{newslave}
        \STATE \texttt{push(newslave, waiting slaves)}
    \ELSIF {no waiting degrees and no working slaves}
        \STATE generate new antichain of degrees
        \IF {no such antichain available}
            \STATE finish up
        \ENDIF
    \ELSE 
        \STATE \textbf{continue}
    \ENDIF
\ENDLOOP
\end{algorithmic}
  \caption{Master algorithm for a distributed Gröbner basis
    computation}
  \label{alg:master}
\end{algorithm}
\begin{algorithm}
\begin{algorithmic}
\LOOP
   \STATE receive message \texttt{msg} from master
    \IF {\texttt{msg} = \texttt{finish}}
        \STATE {\bfseries return}
    \ELSIF {\texttt{msg} = new degree $d$}
        \STATE reduce all S-polynomials in degree d and
               append to Gr\"obner basis
        \STATE compute new S-polynomials based on new basis elements
        \STATE send \texttt{finished degree} to master
    \ENDIF
\ENDLOOP
\end{algorithmic}
  \caption{Slave algorithm for a distributed Gröbner basis
    computation}
  \label{alg:slave}
\end{algorithm}

\section{Experiments}
\label{sec:experiments}

We have implemented the master-slave system described in Section \ref{sec:algorithms} in Sage \cite{sage}, using MPI for Python~\cite{mpi4py,mpi4py2} for distributive computing infrastructure and SQLAlchemy \cite{sqlalchemy} interfacing with a MySQL database \cite{mysql} for an abstraction of a common storage for serialized python objects.

In order to test our implementation, we have used a computational
server running 8 Intel Xeon processors at 2.83 GHz, with a 5M cache,
and a total RAM available of 16G.

We have run test with the Gröbner basis problem $I_3$, and recorded total
wallclock timings, as well as specific timings for the S-polynomial
generation and reduction steps. The problem was run
for each possible number of allocated core (1 to 7 slave processors),
and the server was the entire time otherwise un-utilized.

\begin{figure}[h]
  \centering
  \includegraphics{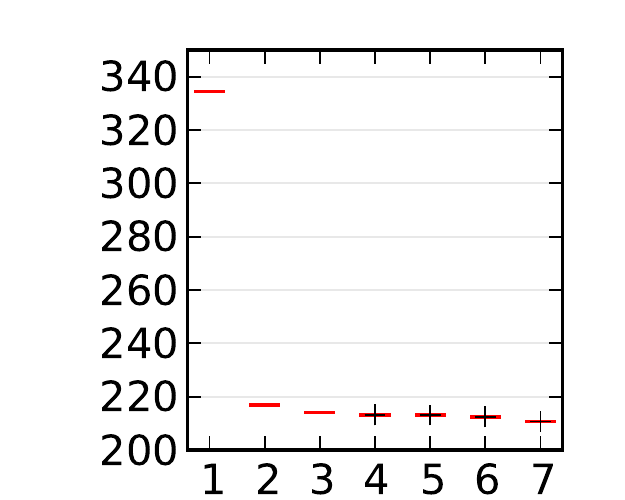}~\includegraphics{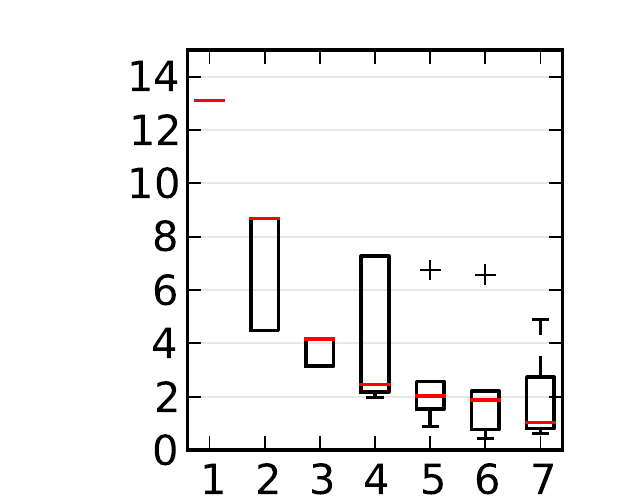}\\
  \includegraphics{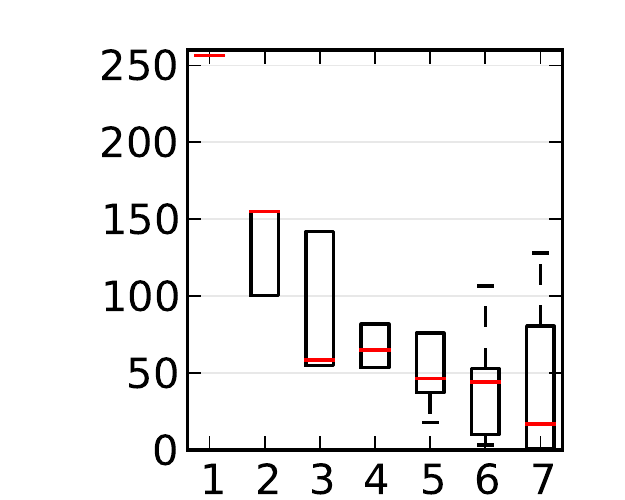}~\includegraphics{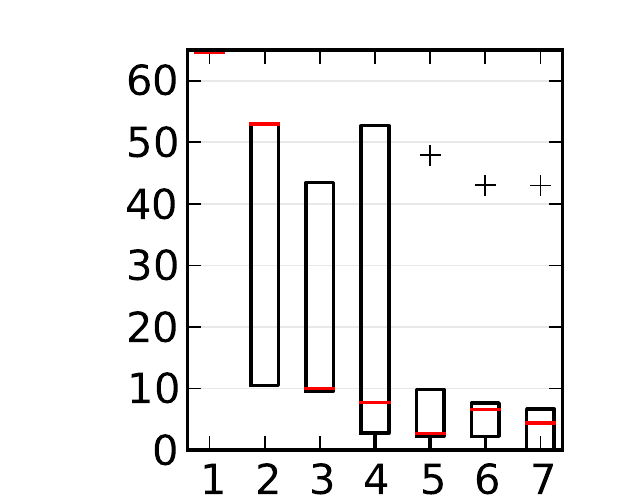}
  \caption{Timings (seconds) for $I_3$ on an 8-core computational server. Timing
    runs were made with between 2 and 8 active processors, and
    the total wallclock times (top left), SQL interaction times (top
    right, the S-polynomial reduction times
    (bottom left), and the S-polynomial generation times (bottom right) were measured.}
  \label{fig:i3timings}
\end{figure}

\begin{figure}[h]
  \centering
  \includegraphics{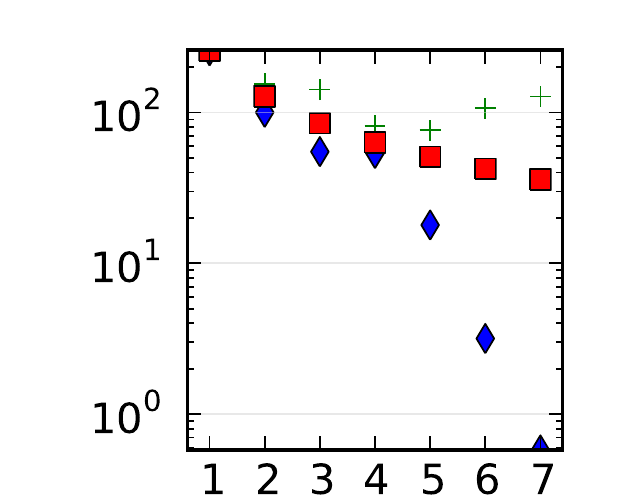}~\includegraphics{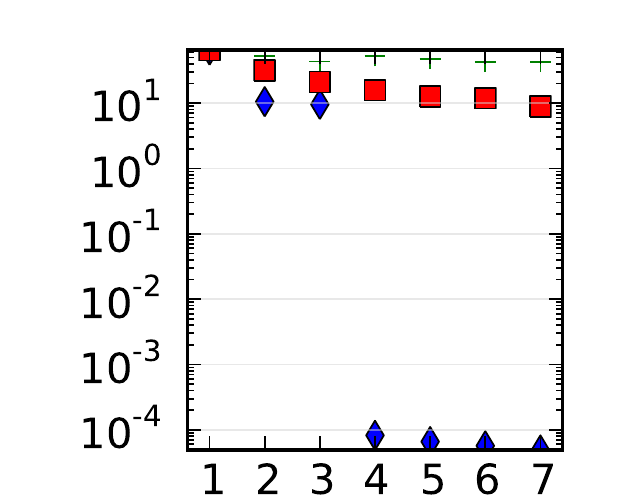}
  \caption{Logarithms of maximal, minimal and average timings (seconds) for
    reduction (left) and
    generation (right) in the $I_3$ computations.}
  \label{fig:i3linreg}
\end{figure}

As can be seen in Figure \ref{fig:i3timings}, parallelization
decreases the wall-clock timings radically compared to single-core
execution (2 processors, with the slave processor doing all work
essentially serially). However, the subsequent decrease in
computational times is less dramatic.

Looking into specific aspects of the computation, we can see that
while the average computational times decrease radically with the
number of available processors, the maximal computation time behaves
much worse. With the reduction step, maximal computation times still
decrease, mostly, with the number of available processors. The
S-polynomial generation step however displays almost constant maximal
generation times along the computation.

Furthermore, compared to the time needed for the algebraic
computations, the relatively slow, database engine mediated storage
and recovery times are almost completely negligible.

These trends are even more clear when we concentrate on only the
maximal, minimal and average computation times, as in Figure
\ref{fig:i3linreg}. We see a proportional decrease in average
computation times, and a radical drop-off in minimal computation
times, which certainly sounds promising. The global behaviour,
however, is dictated by the maximal thread execution times, which are
rather disappointingly behaving throughout.

\section{Conclusions and Future Work}

In conclusion, we have demonstrated that while the parallel
computation of Gröbner bases in general is a problem haunted by the
ghost of data dependency, the lattice structure in an appropriate
choice of multigrading will allow for easy control of
dependencies. Specifically, picking out antichains in the multigrading
lattice gives a demonstrable parallelizability, that saturates the
kind of computing equipment that is easily accessible by researchers
of today.

Furthermore, we have developed our methods publically
accessible,\footnote{\url{http://code.google.com/p/dph-mg-grobner}, the
code used for Section \ref{sec:experiments} can be found in the \texttt{sage}
subdirectory} and released it under the very liberal BSD
license. Hence, with the ease of access to our code and to the Sage
computing system, we try to set the barrier to build further on our
work as low as we possibly can.

However, the techniques we have developed here are somewhat sensitive
to the distribution of workload over the grading lattice: if certain
degrees are disproportionately densely populated, then the
computational burden of an entire Gröbner basis is dictated by the
essentially \emph{serial} computation of the highly populated
degrees. As such, we suspect these methods to work at their very best
in combination with other parallelization techniques.

The Gröbner basis implementation used was a rather naïve one, and we
fully expect speed-ups from sophisticated algorithms to combine
cleanly with the constructions we use. This is something we expect to
examine in future continuation of this project.

There are many places to go from here. We are ourselves interested in
investigating many avenues for the further application of the basic ideas
presented here:
\begin{itemize}
\item Adaptation to state-of-the-art Gröbner basis techniques for
  single processors. Improve the handling of each separate degree,
  potentially subdividing work even finer.
\item Multigraded free modules, and Gröbner bases
  of these; opening up for the use of these methods in computational
  and applied topology, as a computational back bone for multigraded
  persistence.
\item Multigraded free resolutions; opening up for the application of these methods in
  parallelizing computations in homological algebra.
\item Adaptation to non-commutative cases; in particular to use for
  ideals in and modules over quiver algebras.
\item Building on work by Dotsenko and Khoroshkin, and by Dotsenko and
  Vejdemo-Johansson, there is scope to apply this parallelization to
  the computation of Gröbner bases for operads. \cite{dotsenko-khoroshkin,dotsenko-vj}
\end{itemize}

\bibliographystyle{splncs}
\bibliography{dphgrobner}

\def\cprime{$'$}
\begin{thebibliography}{10}

\bibitem{155350}
Chakrabarti, S., Yelick, K.:
\newblock Implementing an irregular application on a distributed memory
  multiprocessor.
\newblock In: PPOPP '93: Proceedings of the fourth ACM SIGPLAN symposium on
  Principles and practice of parallel programming, New York, NY, USA, ACM
  (1993)  169--178

\bibitem{680307}
Vidal, J.P.:
\newblock The computation of gr\"{o}bner bases on a shared memory
  multiprocessor.
\newblock In: DISCO '90: Proceedings of the International Symposium on Design
  and Implementation of Symbolic Computation Systems, London, UK,
  Springer-Verlag (1990)  81--90

\bibitem{siegl}
Siegl, K.:
\newblock A parallel factorization tree {G}r\"obner basis algorithm.
\newblock In: Proceedings of PASCO'94. (1994)

\bibitem{96965}
Bradford, R.:
\newblock A parallelization of the buchberger algorithm.
\newblock In: ISSAC '90: Proceedings of the international symposium on Symbolic
  and algebraic computation, New York, NY, USA, ACM (1990)  296

\bibitem{graebe_lassner}
Gr\"abe, H.G., Lassner, W.:
\newblock A parallel {G}r\"obner factorizer.
\newblock In: Proceedings of PASCO'94. (1994)

\bibitem{10.1109/ICPPW.2004.1328011}
Leykin, A.:
\newblock On parallel computation of {G}r\"obner bases.
\newblock In: Parallel Processing Workshops, International Conference on.
  Volume~0., Los Alamitos, CA, USA, IEEE Computer Society (2004)  160--164

\bibitem{MR2464549}
Mityunin, V.A., Pankrat'ev, E.V.:
\newblock Parallel algorithms for the construction of {G}r\"obner bases.
\newblock Sovrem. Mat. Prilozh. (30, Algebra) (2005)  46--64

\bibitem{MR1624631}
Amrhein, B., B{\"u}ndgen, R., K{\"u}chlin, W.:
\newblock Parallel completion techniques.
\newblock In: Symbolic rewriting techniques ({A}scona, 1995). Volume~15 of
  Progr. Comput. Sci. Appl. Logic.
\newblock Birkh\"auser, Basel (1998)  1--34

\bibitem{MR1635246}
Reeves, A.A.:
\newblock A parallel implementation of {B}uchberger's algorithm over {$Z_p$}
  for {$p\leq 31991$}.
\newblock J. Symbolic Comput. \textbf{26}(2) (1998)  229--244

\bibitem{M2}
Grayson, D.R., Stillman, M.E.:
\newblock Macaulay2, a software system for research in algebraic geometry.
\newblock Available at http://www.math.uiuc.edu/Macaulay2/

\bibitem{MR1324496}
Hreinsd{\'o}ttir, F.:
\newblock A case where choosing a product order makes the calculations of a
  {G}roebner basis much faster.
\newblock J. Symbolic Comput. \textbf{18}(4) (1994)  373--378

\bibitem{MR2251816}
Hreinsd{\'o}ttir, F.:
\newblock An improved term ordering for the ideal of commuting matrices.
\newblock J. Symbolic Comput. \textbf{41}(9) (2006)  999--1003

\bibitem{multiDpersistence}
Carlsson, G., Singh, G., Zomorodian, A.:
\newblock Computing multidimensional persistence.
\newblock In Dong, Y., Du, D.Z., Ibarra, O.H., eds.: ISAAC. Volume 5878 of
  Lecture Notes in Computer Science., Springer (2009)  730--739

\bibitem{sage}
Stein, W.:
\newblock {Sage}: {O}pen {S}ource {M}athematical {S}oftware ({V}ersion 4.3.5).
\newblock The Sage~Group. (2010) {\tt http://www.sagemath.org}.

\bibitem{mpi4py}
Dalc\'{\i}n, L., Paz, R., Storti, M.:
\newblock {MPI} for {P}ython.
\newblock J. Parallel Distrib. Comput. \textbf{65}(9) (2005)  1108 -- 1115

\bibitem{mpi4py2}
Dalc\'{\i}n, L., Paz, R., Storti, M., D'El\'{\i}a, J.:
\newblock {MPI} for {P}ython: Performance improvements and {MPI}-2 extensions.
\newblock J. Parallel Distrib. Comput. \textbf{68}(5) (2008)  655--662

\bibitem{sqlalchemy}
Bayer, M.:
\newblock {SQLAlchemy}.
\newblock \url{http://www.sqlalchemy.org/}

\bibitem{mysql}
Oracle:
\newblock {MySQL}.
\newblock \url{http://mysql.com/}

\bibitem{dotsenko-khoroshkin}
Dotsenko, V., Khoroshkin, A.:
\newblock Gr\"obner bases for operads.
\newblock To appear, Duke Math. Journal (2008)

\bibitem{dotsenko-vj}
Dotsenko, V., Vejdemo-Johansson, M.:
\newblock Implementing {G}r\"obner bases for operads.
\newblock To appear, S\'eminaires et Congr\`es (2009)

\end{thebibliography}

\end{document}